%% file: ex_article.tex
\documentclass[hidelinks,onefignum,onetabnum]{siamart251216}

\input{ex_shared}

\newtheorem{assumption}[theorem]{Assumption}

\ifpdf
\hypersetup{
  pdftitle={},
  pdfauthor={E. C\'aceres and P. Vega}
}
\fi

\usepackage{pdflscape}
\usepackage{amsmath}
\usepackage{multirow}
\usepackage{booktabs}
\input{general_commands}

\begin{document}

\maketitle

% REQUIRED
\begin{abstract}
We introduce an a posteriori error estimator for exactly symmetric mixed finite element discretizations of linear elasticity, based on recasting a Stenberg-type postprocessing scheme as a local residual minimization problem in a discrete dual norm. This construction yields, as a dual variable and at no additional computational cost, a Riesz representative of the associated local residual, from which we build an a posteriori error indicator. We establish a reliability estimate, with the dependence on the Lam\'e parameters tracked explicitly, and a local efficiency estimate, without auxiliary bubble functions, in the standard compressible regime, as well as an alternative reliability estimate, based on a robust stability property and an Oswald averaging operator, with a constant that remains bounded as $\lambda\to\infty$; the local efficiency estimate holds, with the same bounded behavior, uniformly in both regimes. Notably, a single indicator and a single comparison norm serve both regimes, in contrast with existing hypercircle-based estimators, which require a distinct construction for the incompressible limit. Numerical examples, including a benchmark with a known singular solution and one without an analytical solution, validate the theoretical findings.
\end{abstract}

% REQUIRED

% REQUIRED

\section{Introduction}
Mixed finite element methods, see~\cite{BoffiBrezziFortin2013} for a general reference, provide a natural framework for the numerical approximation of linear elasticity problems. Besides yielding a direct approximation of the stress tensor, a quantity of primary physical interest, without the loss of accuracy incurred when recovering it a posteriori by differentiation from a primal displacement approximation, mixed methods are also naturally robust in the incompressible limit, avoiding the locking phenomenon~\cite{BabuskaSuri1992} that plagues standard primal discretizations as the material approaches incompressibility. A central difficulty in constructing such methods is the symmetry of the stress tensor, a physical requirement arising from the balance of angular momentum. This symmetry can be enforced exactly, at the level of the discrete space, or weakly, through a Lagrange multiplier representing the infinitesimal rotation, following the seminal idea of Fraeijs de Veubeke \cite{FraeijsDeVeubeke1977} and later developed by Amara and Thomas \cite{AmaraThomas1979} and Arnold, Brezzi, and Douglas \cite{ArnoldBrezziDouglas1984}, with each approach leading to a distinct family of methods and its own construction challenges; a priori and a posteriori error analyses for both approaches are given, respectively, in \cite{LS2023} and \cite{LS2024}. Several families of exactly symmetric mixed finite elements have been developed, starting with the pioneering work of Watwood and Hartz \cite{WatwoodHartz1968} and Johnson and Mercier \cite{JohnsonMercier1978}, the first stable elements using polynomial stress shape functions due to Arnold and Winther \cite{ArnoldWinther2002}, and later extended by Arnold, Douglas, and Gupta \cite{ArnoldDouglasGupta1984}, Guzm\'an and Neilan \cite{GuzmanNeilan2014}, and Arnold, Awanou, and Winther \cite{ArnoldAwanouWinther2008}; a computational comparison of several such families is given in \cite{CarstensenEigelGedicke2011}.

A recurring theme in mixed-methods analysis is the possibility of improving the accuracy of the displacement approximation through a local postprocessing step performed independently on each element after the global discrete problem has been solved. For mixed methods in elasticity, this idea dates back to Stenberg \cite{Stenberg1988} and yields a superconvergent postprocessed displacement at negligible additional computational cost, since the postprocessing is entirely local. We use the a priori error analysis of Lederer and Stenberg \cite{LS2023} as our starting point for exactly symmetric mixed methods. Our first contribution is to show that this postprocessing scheme can be recast, following the methodology introduced by Muga, Rojas, and Vega \cite{MugaRojasVega2023} for scalar diffusion problems, as a local residual minimization problem in a discrete dual norm. Establishing this equivalence is not enough, on its own, to justify the reformulation: what makes it useful is that it produces, as a dual variable and at no additional computational cost, a residual representative in the Riesz sense. This residual representative is our second contribution: we use it to construct an a posteriori error indicator, for which we prove reliability and local efficiency estimates, with the dependence on the Lam\'e parameters tracked explicitly, in the standard, compressible regime, without resorting to the auxiliary bubble functions typically required by classical residual-based estimators, such as those of Carstensen and Dolzmann \cite{CarstensenDolzmann1998} and Lonsing and Verf\"urth \cite{LonsingVerfurth2004}. Our third and final contribution addresses a limitation shared by this construction and the underlying methodology: the loss of control over the stress trace in the compliance norm as the material approaches the incompressible limit. By combining a stability property of the continuous problem, uniform in the Lam\'e parameters, with an Oswald averaging operator applied to the postprocessed displacement, we derive an alternative reliability estimate, together with a matching local efficiency estimate, with a constant that remains bounded as $\lambda\to\infty$. Notably, both estimates are given in terms of the same indicator and the same comparison norm, unlike the hypercircle-based estimator of \cite{LS2023}, which requires a second, distinct construction for the incompressible limit, since the norm intrinsic to the hypercircle identity is no longer a norm in this regime. A forthcoming companion paper, following the same methodology, addresses the case of weakly symmetric mixed finite elements.

\subsection{Related work}

A posteriori error estimation for finite element methods is a mature field; see~\cite{AinsworthOden2000} for a comprehensive treatment. For mixed methods in elasticity specifically, Lovadina and Stenberg~\cite{LovadinaStenberg2006} were the first to show that the superconvergent postprocessing of Stenberg~\cite{Stenberg1991} yields a simple estimate. For methods with exactly symmetric stresses, Lederer and Stenberg \cite{LS2023} derive estimates based on the Prager-Synge hypercircle principle, together with an alternative estimate that remains valid, with a constant that remains bounded as $\lambda\to\infty$; a companion analysis for weakly symmetric methods is given in \cite{LS2024}. Classical residual-based estimators for exactly symmetric mixed elements, requiring an efficient approximation of the skew-symmetric part of the gradient, were developed by Carstensen and Gedicke \cite{CarstensenGedicke2016}, and later improved to an explicit estimator depending only on the symmetric stress, without such an approximation, by Carstensen, Gallistl, and Gedicke \cite{CarstensenGallistlGedicke2019}; both, however, rely on auxiliary bubble functions in the efficiency analysis, which the present construction avoids entirely. Earlier a posteriori estimates for mixed elasticity, not based on postprocessing, include the work of Carstensen and Dolzmann \cite{CarstensenDolzmann1998} and Lonsing and Verf\"urth \cite{LonsingVerfurth2004}.

The use of residual minimization to construct both the numerical approximation and a built-in a posteriori error indicator has been developed in the context of least-squares finite element methods and discontinuous Petrov-Galerkin (DPG) methods, introduced by Demkowicz and Gopalakrishnan \cite{DemkowiczGopalakrishnan2011}, who show that the resulting discretization can be equivalently interpreted as a minimum-residual method in a dual norm; see also \cite{CarstensenDemkowiczGopalakrishnan2014}. DPG methods provide, in particular, a built-in residual representative for elasticity: Bramwell, Demkowicz, Gopalakrishnan, and Qiu \cite{BramwellDemkowiczGopalakrishnanQiu2012} develop a locking-free DPG method with exactly symmetric stresses, and Keith, Fuentes, and Demkowicz \cite{KeithFuentesDemkowicz2016} and Fuentes, Keith, Demkowicz, and Le Tallec \cite{FuentesKeithDemkowiczLeTallec2017} analyze related variational formulations. Closer to the present work, Muga, Rojas, and Vega \cite{MugaRojasVega2023} combined the postprocessing technique of Stenberg \cite{Stenberg1991} with local residual minimization for scalar diffusion problems, obtaining a superconvergent postprocessed solution together with a reliable and efficient a posteriori error indicator. This methodology was later extended by Camargo, Rojas, and Vega \cite{CamargoRojasVega2025} to hybridizable discontinuous Galerkin discretizations of the Helmholtz equation. The present work extends the same methodology to linear elasticity with exactly symmetric stresses, and a forthcoming companion paper addresses the case of weakly symmetric stresses.

\subsection{Outline of the paper}
The remainder of this paper is organized as follows. Section~\ref{sec:problem_statement} introduces the model problem, the family of exactly symmetric mixed finite element methods under consideration, and the associated a priori error estimate, together with the postprocessing scheme of Stenberg. Section~\ref{sec:postproc_minres} recasts this postprocessing scheme as a local residual minimization problem and establishes its equivalence with the original construction. Section~\ref{sec:aposteriori} introduces an a posteriori error indicator based on the resulting residual and establishes its reliability and local efficiency, both in the compressible regime and, with a constant that remains bounded as $\lambda\to\infty$, in the incompressible limit. Numerical examples illustrating these theoretical findings are reported in Section~\ref{sec:numerical_examples}.

\section{Problem statement and discretization}\label{sec:problem_statement}

\subsection{The equations of elasticity}\label{sec:elasticity_eqns}

Let $\Omega\subset\mathbb{R}^d$, $d=2,3$, be a polygonal or polyhedral domain, with boundary $\Gamma\eq\partial\Omega$ split into two disjoint sets $\Gamma_D$ and $\Gamma_N$, with $|\Gamma_D|>0$. The linear elasticity problem, where the constitutive relationship between the stress tensor and the strain tensor $\varepsilon(\bu)\eq\frac{1}{2}(\nabla\bu+(\nabla\bu)^\T)$ is linear, in mixed form consists of finding the displacement $\bu=(u_1,\ldots,u_d):\Omega\to\mathbb{R}^d$ and the symmetric\footnote{That is, $\sigma_{ij}=\sigma_{ji}$ for $i,j=1,\ldots,d$.} stress tensor $\bsigma=(\sigma_{ij})_{i,j=1}^d:\Omega\to\mathbb{R}^{d\times d}$ such that
\begin{subequations}\label{eq:mixed_elasticity}
\begin{alignat}{2}
    \CC\bsigma-\varepsilon(\bu)&=\bzero &&\quad\text{in }\Omega,\\
    -\ddiv\bsigma&=\bff &&\quad\text{in }\Omega,\\
    \bu&=\bzero &&\quad\text{on }\Gamma_D,\\
    \bsigma\bn&=\bg &&\quad\text{on }\Gamma_N,
\end{alignat}
\end{subequations}
where $\bff$ is a body load and $\bg$ is a traction on the boundary part $\Gamma_N$. We consider an isotropic material, adopting the plane strain hypothesis (i.e., zero strain in the out-of-plane direction) when $\Omega$ is polygonal. For an isotropic material, the compliance tensor $\CC$ is given by
\begin{equation}
    \CC\btau=\frac{1}{2\mu}\Big(\btau-\frac{\lambda}{2\mu+d\lambda}\mathrm{tr}(\btau)I\Big),
\end{equation}
where $\mu$ and $\lambda$ denote the Lam\'e parameters. We denote by $\mathcal{A}\eq\CC^{-1}$ the elasticity tensor, explicitly given by
\begin{equation}
    \mathcal{A}\btau=2\mu\btau+\lambda\,\mathrm{tr}(\btau)I.
\end{equation}
Note that, as $\lambda\to\infty$, the material approaches the incompressible limit, characterized by the constraint $\ddiv\bu=\bzero$.

Since the Lam\'e parameters $\mu,\lambda$ will appear throughout the error estimates that follow, we fix here the notation used to track their dependence: we write $A\lesssim B$ (and $A\gtrsim B$) when there exists a positive constant $C$, independent of the mesh parameter $h$, such that $A\le CB$ (respectively, $A\ge CB$). Unless stated otherwise, $C$ may depend on $\mu,\lambda$; whenever a stated estimate holds with $C$ independent of these parameters as well, this is indicated explicitly in the statement.

\subsection{Variational form}

The system \eqref{eq:mixed_elasticity} admits a primal mixed variational formulation: find $(\bsigma,\bu)\in [L^2(\Omega)]^{d\times d}_{\text{sym}}\times [H^1_D(\Omega)]^d$ such that
\begin{equation}\label{eq:primal_mixed_bilinear}
    \CB(\bsigma,\bu;\btau,\bv)=(\boldsymbol{f},\bv)+\langle\boldsymbol{g},\bv\rangle_{\Gamma_N} \quad \forall(\btau,\bv)\in [L^2(\Omega)]^{d\times d}_{\text{sym}}\times [H^1_D(\Omega)]^d,
\end{equation}
with
\begin{equation}\label{eq:primal_mixed_form}
    \CB(\bsigma,\bu;\btau,\bv)\eq(\CC\bsigma,\btau)-(\varepsilon(\bu),\btau)-(\varepsilon(\bv),\bsigma),
\end{equation}
and associated energy norms
\begin{equation}\label{eq:energy-norms}
    \|\btau\|_\CC^2\eq(\CC\btau,\btau), \qquad \|\varepsilon(\bv)\|_\CA^2\eq(\CA\varepsilon(\bv),\varepsilon(\bv)),
\end{equation}
which is the setting in which stability of the continuous problem is established, in the energy norm $\big(\|\btau\|_\CC^2+\|\varepsilon(\bv)\|_\CA^2\big)^{1/2}$, with a known, sharp constant~\cite{HannukainenStenbergVohralik2012}. In the incompressible limit, $\|\cdot\|_\CC$ is no longer a norm, and this estimate degenerates; a stronger, uniform version of this stability estimate, needed for the a posteriori analysis in that limit, is stated in Theorem~\ref{thm:robust-stability}. Equation \eqref{eq:primal_mixed_bilinear} can be integrated by parts to yield an equivalent dual mixed formulation \cite[Section 2.2]{LS2023}, transferring regularity from the displacement to the stress; this is the formulation on which we base the finite element discretization.

The dual mixed variational formulation reads: find $(\bsigma,\bu)\in H_g(\ddiv;\Omega)\times[L^2(\Omega)]^d$ such that
\begin{equation}\label{eq:variationalform_mixed_elasticity}
    \CM(\bsigma,\bu;\btau,\bv)=-(\bff,\bv)\quad\forall(\btau,\bv)\in H_0(\ddiv;\Omega)\times[L^2(\Omega)]^d,
\end{equation}
with
\begin{equation}
    \CM(\bsigma,\bu;\btau,\bv)\eq(\CC\bsigma,\btau)+(\bu,\ddiv\btau)+(\bv,\ddiv\bsigma),
\end{equation}
where
\begin{align}
    H_g(\ddiv;\Omega)&\eq\{\btau\in H(\ddiv;\Omega)\mid \btau\bn|_{\Gamma_N}=\bg\},\\
    H_0(\ddiv;\Omega)&\eq\{\btau\in H(\ddiv;\Omega)\mid \btau\bn|_{\Gamma_N}=\bzero\}.
\end{align}

\subsection{Exactly symmetric mixed finite element methods}

Following \cite{LS2023}, we present a unified framework for mixed finite element methods, based on the variational formulation \eqref{eq:variationalform_mixed_elasticity}. Let $V_h\subset[L^2(\Omega)]^d$ and $S_h\subset H(\ddiv;\Omega)$ be piecewise polynomial subspaces. We consider the linear triangular method of \emph{Johnson-Mercier} (JM) \cite{JohnsonMercier1978,WatwoodHartz1968}, the triangular family of \emph{Arnold-Douglas-Gupta} (ADG) \cite{ArnoldDouglasGupta1984}, the triangular family of \emph{Guzm\'an-Neilan} (GN) \cite{GuzmanNeilan2014}, and the tetrahedral family of \emph{Arnold-Awanou-Winther} (AAW) \cite{ArnoldAwanouWinther2008}.
We use $\CT_h$ to denote the underlying triangular or tetrahedral mesh. Each family depends on a polynomial degree $k\geq 2$ and shares the same local displacement space $V(K)\eq[P_{k-1}(K)]^d$, giving rise to the global space
\begin{equation}\label{eq:V_h}
    V_h\eq\{\bv\in[L^2(\Omega)]^d\mid \bv|_K\in V(K)\ \forall K\in\CT_h\}.
\end{equation}
For the JM method, $V(K)$ is instead taken with $k=2$, i.e., $V(K)=[P_1(K)]^d$, corresponding to discontinuous piecewise linear polynomials.

Rather than reproducing the explicit construction of the local stress spaces $S(K)$, we recall from \cite{LS2023} only the two properties needed in what follows, and define the associated global space
\begin{equation}\label{eq:S_h}
    S_h\eq\{\btau\in H(\ddiv;\Omega)\mid \btau|_K\in S(K)\ \forall K\in\CT_h\}.
\end{equation}
The approximation order of $S(K)$ is ensured by the inclusion $[P_k(K)]^{n\times n}_{\text{sym}}\subset S(K)$, while its local degrees of freedom are given by the moments
\begin{subequations}
\begin{alignat}{2}
    \int_K\btau:\varepsilon(\bv) &\quad \forall\bv\in[P_{k-1}(K)]^d,\\
    \int_E\btau\bn\cdot\bv &\quad \forall\bv\in[P_k(E)]^d,\ \text{for each edge or face $E$ of $K$}.
\end{alignat}
\end{subequations}
For the JM method, the value $k=1$ in the formulas above should be understood formally, as a schematic device to fit this method within the unified indexing of the other families, rather than a literal description of its degrees of freedom, which follow the original construction of~\cite {JohnsonMercier1978} instead.

As a further consequence of this non-standard construction, valid for every family except JM, $\ddiv\btau\in V_h$ whenever $\btau\in S_h$; consequently,
\begin{equation}\label{eq:l2_projection_property}
    (\ddiv\btau,\bv-P_h\bv)=0\quad\forall\btau\in S_h,
\end{equation}
with $P_h:[L^2(\Omega)]^d\to V_h$ the $L^2$-projection. This property does not hold for JM; nevertheless, a projection satisfying \eqref{eq:l2_projection_property} can still be constructed for that method, see \cite[Section 5]{JohnsonMercier1978} and \cite[Lemma~4.3]{PitkarantaStenberg1983}.

For each edge or face $E\subset\Gamma_N$, let $Q_E:[L^2(E)]^d\to[P_k(E)]^d$ denote the $L^2$-projection, and let $Q_h$ be defined edge/face-wise by $Q_h|_E\eq Q_E$.

The discrete trial and test spaces are then given by
\begin{align}
    S_h^g&\eq\{\btau\in S_h\mid \btau\bn=Q_hg\text{ on }\Gamma_N\},\\
    S_h^0&\eq\{\btau\in S_h\mid \btau\bn=\bzero\text{ on }\Gamma_N\}.
\end{align}

Now, we are in a position to state the discrete problem: find $(\bsigma_h,\bu_h)\in S_h^g\times V_h$ such that
\begin{equation}\label{eq:discrete_problem}
    \CM(\bsigma_h,\bu_h;\btau,\bv)+(\bff,\bv)=0\quad\forall(\btau,\bv)\in S_h^0\times V_h.
\end{equation}

For the analysis of the discrete problem \eqref{eq:discrete_problem}, we make use of the broken energy norm for the displacement,
\begin{equation}
    \|\bv\|_h^2\eq\sum_{K\in\CT_h}\|\varepsilon(\bv)\|_K^2+\sum_{E\in\CE_h^0}h_E^{-1}\|\jmp{\bv}\|_E^2+\sum_{E\in\CE_h^D}h_E^{-1}\|\bv\|_E^2,\quad\bv\in V_h,
\end{equation}
where, for an interior edge/face $E\in\CE_h^0$ shared by two elements $K^+,K^-\in\CT_h$ with outward unit normals $n^+,n^-$, the jump of $\bv$ across $E$ is defined by
\begin{equation}
    \jmp{\bv}|_E\eq\bv|_{K^+}-\bv|_{K^-}.
\end{equation}
Together with the $\mu$-weighted $L^2$-norm for the stress, this defines the norm on the product space $S_h^0\times V_h$ given by
\begin{equation}
    \|(\boldsymbol{\tau},\bv)\|_h^2\eq\mu^{-1}\|\boldsymbol{\tau}\|_0^2+\mu\|\bv\|_h^2.
\end{equation}

The stability of the discrete problem \eqref{eq:discrete_problem}, with respect to $\|(\cdot,\cdot)\|_h$, is established in \cite[Section 4]{LS2023}; since our focus lies elsewhere, we omit this analysis entirely and recall only the resulting a priori error estimate, which will be used in Section~\ref{sec:postproc_stenberg} when discussing the postprocessing of the displacement.

\begin{theorem}[A priori error estimate]\label{thm:a_priori}
    Let $(\bu,\bsigma)$ solve \eqref{eq:mixed_elasticity}, and let\\ $(\bu_h,\bsigma_h)\in V_h\times S_h^g$ solve \eqref{eq:discrete_problem}. It holds that
    \begin{align}
        \mu^{-1/2}\|\bsigma-\bsigma_h\|_0\,+\,&\mu^{1/2}\|P_h\bu-\bu_h\|_h\nonumber\\
        &\lesssim
        \mu^{-1/2}\Bigg(\inf_{\btau\in S_h^g}\|\bsigma-\btau\|_0
        +\Bigg(\sum_{K\in\CT_h}h_K^2\|\bff-\bff_h\|_K^2\Bigg)^{\!1/2\,}\Bigg),
    \end{align}
    where $\bff_h$ is any piecewise polynomial approximation of $\bff$.
\end{theorem}
\begin{proof}
    See \cite[Theorem 3]{LS2023}.
\end{proof}

For a sufficiently smooth exact solution, Theorem~\ref{thm:a_priori} yields the expected optimal convergence rate $\|\bsigma-\bsigma_h\|_0=\mathcal{O}(h^{k+1})$ for the stress.

\begin{remark}[superconvergence of the projected displacement error]
    Combining Theorem~\ref{thm:a_priori} with the triangle inequality $\|\bu-\bu_h\|_h\leq\|\bu-P_h\bu\|_h+\|P_h\bu-\bu_h\|_h$ and the approximation properties of $P_h$, one obtains $\|\bu-\bu_h\|_h=\mathcal{O}(h^{k-1})$ and $\|P_h\bu-\bu_h\|_h=\mathcal{O}(h^{k+1})$; a superconvergent rate (except for the JM method, for which $\|\bu-\bu_h\|_h=\mathcal{O}(h)$ and $\|P_h\bu-\bu_h\|_h=\mathcal{O}(h^2)$). This superconvergence is precisely the property exploited by the postprocessing scheme introduced in the next section.
\end{remark}

\subsection{Postprocessing scheme for the displacement}\label{sec:postproc_stenberg}

Let $(\bsigma_h,\bu_h)\in S_h^g\times V_h$ be the solution of \eqref{eq:discrete_problem}. The idea of postprocessing the displacement in mixed methods for elasticity was originally introduced by Stenberg \cite{Stenberg1988}. This postprocessing scheme is applicable regardless of whether the stress symmetry is imposed weakly or exactly; here, following \cite{LS2023}, it is built from the discrete solution $(\bsigma_h,\bu_h)\in S_h^g\times V_h$ of \eqref{eq:discrete_problem}, so that the resulting superconvergence relies on the a priori error estimate of Theorem~\ref{thm:a_priori}, which is specific to the case of exact stress symmetry. Let
\begin{equation}\label{eq:V_star_K}
    V^{l}(K)\eq[P_{l}(K)]^d,
\end{equation}
with $l=k+1$ for ADG, GN, and AAW, and $l=k$ for JM, and let $V_h^{l}$ denote its associated global space,
\begin{equation}\label{eq:V_star_h}
    V_h^{l}\eq\{\bv\in[L^2(\Omega)]^d\mid \bv|_K\in V^{l}(K)\quad\forall K\in\CT_h\}.
\end{equation}

We look for $\widetilde{\bu}_h\in V_h^{l}$ satisfying the following equations on each $K\in\CT_h$:
\begin{subequations}\label{eq:local_postproc}
\begin{alignat}{2}
    (\varepsilon(\widetilde{\bu}_h),\varepsilon(\bv_K))_K&=(\CC\bsigma_h,\varepsilon(\bv_K))_K &&\quad\forall\bv_K\in(I-P_h)V^{l}_h\big|_K,\label{eq:local_postproc_a}\\
    (\widetilde{\bu}_h,\bv_K)_K&=(\bu_h,\bv_K)_K &&\quad\forall\bv_K\in P_hV^{l}_h\big|_K.\label{eq:local_postproc_b}
\end{alignat}
\end{subequations}

The main purpose of introducing the postprocessed approximation $\widetilde{\bu}_h$ is to improve the quality of the displacement field approximation. One of the advantages of \eqref{eq:local_postproc} is its local nature, enabling computation at negligible computational cost compared to resolving the original approximation scheme, with parallel computing.

The next theorem shows an a priori error estimate for $\bu-\widetilde{\bu}_h$. In contrast to \cite[Lemma 4]{LS2023}, we keep the term $\|P_h\bu-\bu_h\|_h$ explicit in the estimate, to emphasize the dependence of the superconvergence of $\widetilde{\bu}_h$ on the a priori error estimates for $\bsigma-\bsigma_h$ and $P_h\bu-\bu_h$ given in Theorem~\ref{thm:a_priori}. We also let $P_h^l:[L^2(\Omega)]^d\to V_h^l$ denote the $L^2$-projection onto $V_h^l$.

\begin{theorem}\label{thm:postproc_bound}
    Let $(\bu,\bsigma)$ solve \eqref{eq:mixed_elasticity}; let $(\bu_h,\bsigma_h)\in V_h\times S_h^g$ solve \eqref{eq:discrete_problem}; and let $\widetilde{\bu}_h\in V_h^{l}$ be defined by \eqref{eq:local_postproc}. Then, it holds that
    \begin{equation}\label{eq:apriori_postproc}
        \|\bu-\widetilde{\bu}_h\|_h \lesssim \|\bu-P_h^{l}\bu\|_h + \mu^{-1}\|\bsigma-\bsigma_h\|_0 + \|P_h\bu-\bu_h\|_h.
    \end{equation}
\end{theorem}
\begin{proof}
    By the triangle inequality and the identity $$P_h^{l}\bu-\widetilde{\bu}_h=(P_h^{l}-P_h)(\bu-\widetilde{\bu}_h)+(P_h\bu-\bu_h),$$ which follows from $P_h\widetilde{\bu}_h=\bu_h$, it suffices to bound $\|(P_h^{l}-P_h)(\bu-\widetilde{\bu}_h)\|_h$. This is done exactly as in \cite[proof of Lemma 4]{LS2023}, yielding $$\|(P_h^{l}-P_h)(\bu-\widetilde{\bu}_h)\|_h \lesssim \|\bu-P_h^{l}\bu\|_h+\mu^{-1}\|\bsigma-\bsigma_h\|_0.$$ The claim follows by combining these estimates.
\end{proof}

Combining Theorem~\ref{thm:postproc_bound} with Theorem~\ref{thm:a_priori}, we conclude that, for a sufficiently smooth exact solution, $\|\bu-\widetilde{\bu}_h\|_h=\mathcal{O}(h^{l})$, i.e., $\widetilde{\bu}_h$ is superconvergent.

\section{Postprocessing scheme via local residual minimization}\label{sec:postproc_minres}
Following the methodology introduced by \cite{MugaRojasVega2023}, we develop a postprocessed approximation based on local residual minimization for the elasticity setting, taking the postprocessing scheme \eqref{eq:local_postproc} as its starting point. Adapting this methodology to the vector-valued setting of elasticity requires replacing the full gradient of the scalar case with the symmetric gradient $\varepsilon(\cdot)$ throughout, both in the discrete dual norm used below and in the resulting local mixed system. As we show below, this construction naturally yields, as a dual variable, the Riesz representative of the associated residual, which will serve as the main ingredient of the a posteriori error analysis developed in Section~\ref{sec:aposteriori}.

We define a new postprocessed approximation $\bw_h$ through the following set of local residual minimization problems: find $\bw_h\in V^{l}_h$ such that for each $K\in\CT_h$,
\begin{subequations}\label{eq:local_postproc_minres}
\begin{align}
    \bw_h|_K&=\argmin_{\bv_K\in(I-P_h)V^{l}_h\big|_K}\frac{1}{2}\left\|\CC\bsigma_h-\varepsilon(\bv_K)\right\|_{*,K}^2\,,\label{eq:local_postproc_minres_a}\\
    (\bw_h,\bv_K)_K&=(\bu_h,\bv_K)_K\qquad\qquad\quad\,\forall\bv_K\in P_hV^{l}_h\big|_K,\label{eq:local_postproc_minres_b}
\end{align}
\end{subequations}
where $\|\cdot\|_{*,K}$ is a discrete dual norm defined by
\begin{align}
    \|\cdot\|_{*,K}\eq\sup_{\bv_K\in(I-P_h)V^{l}_h\big|_K}\frac{(\ \cdot\ ,\varepsilon(\bv_K))_K}{\|\varepsilon(\bv_K)\|_K}.
\end{align}

As usual in minimum-residual methods, we can write~\eqref{eq:local_postproc_minres} equivalently as the following local mixed linear system (see~\cite{CohenDahmenWelper2012}): find $\br_h\in(I-P_h)V_h^{l+1}$ and $\bw_h\in V_h^{l}$ such that for each $K\in\CT_h$,
\begin{subequations}\label{eq:local_postproc_mixed}
\begin{alignat}{2}
(\varepsilon(\br_h),\varepsilon(\bv_K))_K+(\varepsilon(\bw_h),\varepsilon(\bv_K))_K&=(\CC\bsigma_h,\varepsilon(\bv_K))_K &&\quad\forall\bv_K\in(I-P_h)V^{l+1}_h\big|_K,\label{eq:local_postproc_mixed_a}\\
(\varepsilon(\bv_K),\varepsilon(\br_h))_K&=0 &&\quad\forall\bv_K\in(I-P_h)V^{l}_h\big|_K,\label{eq:local_postproc_mixed_b}\\
(\bw_h,\bv_K)_K&=(\bu_h,\bv_K)_K &&\quad\forall\bv_K\in P_h V^{l}_h\big|_K.\label{eq:local_postproc_mixed_c}
\end{alignat}
\end{subequations}

The next result shows the relationship between the postprocessed approximation $\bw_h$ and the original postprocessed solution $\widetilde{\bu}_h$. Although $\bw_h$ is constructed via residual minimization in a discrete dual norm of the test space and $\widetilde{\bu}_h$ is obtained as the solution of an elliptic problem, both postprocessed approximations coincide.
\begin{proposition}\label{prop:postproc_equiv}
    Let $\widetilde{\bu}_h\in V_h^{l}$ solve \eqref{eq:local_postproc}, and let $\bw_h\in V_h^{l}$ solve the residual minimization problem \eqref{eq:local_postproc_minres}. Then, $\bw_h=\widetilde{\bu}_h$.
\end{proposition}
\begin{proof}
    Since $(I-P_h)V^{l}_h\big|_K\subset(I-P_h)V^{l+1}_h\big|_K$, for each $K\in\CT_h$, we may test \eqref{eq:local_postproc_mixed_a} with $\bv_K\in(I-P_h)V^{l}_h\big|_K$ and then, thanks to \eqref{eq:local_postproc_mixed_b}, we conclude that $\bw_h$ solves \eqref{eq:local_postproc}. The result follows since \eqref{eq:local_postproc} has a unique solution.
\end{proof}
Hereafter, we denote both postprocessed approximations by $\bw_h$. As a consequence of Proposition~\ref{prop:postproc_equiv}, we obtain the following.
\begin{corollary}\label{cor:postproc_bound_minres}
    The postprocessed approximation $\bw_h\in V_h^{l}$ also satisfies the a priori error estimate \eqref{eq:apriori_postproc} given in Theorem~\ref{thm:postproc_bound}.
\end{corollary}

\section{A posteriori error analysis}\label{sec:aposteriori}

We now turn to the a posteriori error analysis, based on the residual representative $\br_h$ obtained from the postprocessing scheme of Section~\ref{sec:postproc_minres}.

Before introducing the a posteriori error indicator, we define the local counterpart of $\|\cdot\|_h$. For $\bv\in V_h$ and $K\in\CT_h$, let
\begin{equation}
    \|\bv\|_{h,K}^2\eq\|\varepsilon(\bv)\|_K^2+\frac{1}{2}\sum_{E\in\CE_K^0}h_E^{-1}\|\jmp{\bv}\|_E^2+\sum_{E\in\CE_K^D}h_E^{-1}\|\bv\|_E^2,
\end{equation}
so that $\|\bv\|_h^2=\sum_{K\in\CT_h}\|\bv\|_{h,K}^2$.

Now, we define the local error indicator $\eta_K$ by
\begin{align}\label{eq:etaK}
    \eta_K^2\eq\mu\|\varepsilon(\br_h)\|_K^2&+\mu\left\|\CC\bsigma_h-\varepsilon(\bw_h)\right\|_K^2\\
    &+\frac{\mu}{2}\sum_{E\in\CE_K^0}h_E^{-1}\|\jmp{\bw_h}\|_E^2+\mu\sum_{E\in\CE_K^D}h_E^{-1}\|\bw_h\|_E^2,\nonumber
\end{align}
and the global indicator $\eta_h\eq\big(\sum_{K\in\CT_h}\eta_K^2\big)^{1/2}$.

To state our reliability estimate, we make a saturation assumption, for which we introduce the following auxiliary problem: Find $\bz_h\in V_h^{l+1}$ such that for each $K\in\CT_h$,
\begin{subequations}\label{eq:local_postproc_aux}
\begin{alignat}{2}
    (\varepsilon(\bz_h),\varepsilon(\bv_K))_K&=(\CC\bsigma_h,\varepsilon(\bv_K))_K &&\quad\forall\bv_K\in(I-P_h)V^{l+1}_h\big|_K,\label{eq:local_postproc_aux_a}\\
    (\bz_h,\bv_K)_K&=(\bu_h,\bv_K)_K &&\quad\forall\bv_K\in P_h V^{l+1}_h\big|_K.\label{eq:local_postproc_aux_b}
\end{alignat}
\end{subequations}

Now, we introduce the saturation assumption.
\begin{assumption}[saturation]\label{ass:saturation}
    Let $(\br_h,\bw_h)\in V_h^{l+1}\times V_h^{l}$ solve \eqref{eq:local_postproc_mixed}; and let $\bz_h\in V_h^{l+1}$ solve \eqref{eq:local_postproc_aux}. There exists a real number $\delta\in[0,1)$, independent of $h$ and of $\mu,\lambda$ within any fixed range bounded away from the incompressible limit, such that
    \begin{equation}
        \|\varepsilon(\bu-\bz_h)\|_0\leq\delta\|\varepsilon(\bu-\bw_h)\|_0.
    \end{equation}
\end{assumption}

\begin{remark}[Saturation constant near the incompressible limit]
The independence of $\delta$ from $\mu,\lambda$ postulated above is verified numerically in Section~\ref{sec:num-smooth}, by computing the ratio $\|\varepsilon(\bu-\bz_h)\|_0/\|\varepsilon(\bu-\bw_h)\|_0$ directly for the smooth example, across the same range of Poisson ratios considered there.
\end{remark}

The following auxiliary result, together with Assumption~\ref{ass:saturation}, will be used to establish the reliability of the error indicator.

\begin{lemma}\label{lemma:rel}
    Let $(\br_h,\bw_h)\in V_h^{l+1}\times V_h^{l}$ solve \eqref{eq:local_postproc_mixed}; and let $\bz_h\in V_h^{l+1}$ solve \eqref{eq:local_postproc_aux}. Then, for all $K\in\CT_h$, the following holds true:
    \begin{equation}
        \|\varepsilon(\bz_h-\bw_h)\|_K=\|\varepsilon(\br_h)\|_K.
    \end{equation}
\end{lemma}

\begin{proof}
    Combining \eqref{eq:local_postproc_mixed_a} with \eqref{eq:local_postproc_aux_a}, we get
    \begin{align*}
        (\varepsilon(\bz_h-\bw_h),\varepsilon(\bv_K))_K&=-(\varepsilon(\bw_h),\varepsilon(\bv_K))_K+(\varepsilon(\bz_h),\varepsilon(\bv_K))_K\\
        &=(\varepsilon(\br_h),\varepsilon(\bv_K))_K-(\CC\bsigma_h,\varepsilon(\bv_K))_K+(\varepsilon(\bz_h),\varepsilon(\bv_K))_K\\
        &=(\varepsilon(\br_h),\varepsilon(\bv_K))_K
    \end{align*}
    for all $\bv_K\in(I-P_h)V^{l+1}_h\big|_K$. Since $(I-P_h)V^{l+1}_h\big|_K\cap RM(K)=\{\bzero\}$, where $RM(K)$ denotes the space of rigid-body motions on $K$, by the second Korn inequality~\cite[Theorem~42.10]{ErnGuermondII}, $\|\varepsilon(\cdot)\|_K$ defines a genuine norm on $(I-P_h)V^{l+1}_h\big|_K$, so that
    \begin{align*}
        \|\varepsilon(\bz_h-\bw_h)\|_K&=\sup_{\bv_K\in(I-
        P_h)V^{l+1}_h\big|_K\setminus\{0\}}\frac{(\varepsilon(\bz_h-\bw_h),\varepsilon(\bv_K))_K}{\|\varepsilon(\bv_K)\|_K}\\
        &=\sup_{\bv_K\in(I-P_h)V^{l+1}_h\big|_K\setminus\{0\}}\frac{(\varepsilon(\br_h),\varepsilon(\bv_K))_K}{\|\varepsilon(\bv_K)\|_K}\\
        &=\|\varepsilon(\br_h)\|_K.
    \end{align*}
\end{proof}

For the reliability estimate below, we will also use the following bounds on the elasticity and compliance tensors. Since $\mathcal{A}\btau=2\mu\btau+\lambda\,\text{tr}(\btau)I$, the triangle inequality, together with $\|\text{tr}(\btau)I\|_0\le d\|\btau\|_0$, gives
\begin{equation}\label{eq:A_bound}
    \|\mathcal{A}\btau\|_0 \leq (2\mu+d\lambda)\|\btau\|_0 \quad \forall\btau\in [L^2(\Omega)]^{d\times d}_{\text{sym}}.
\end{equation}
Similarly, $\CC\btau=\frac{1}{2\mu}\big(\btau-\frac{\lambda}{2\mu+d\lambda}\text{tr}(\btau)I\big)$, the bound on $\|\text{tr}(\btau)I\|_0$, and $\frac{2\mu+2d\lambda}{2\mu+d\lambda}\le2$, give
\begin{equation}\label{eq:C_bound}
    \|\CC\btau\|_0 \leq \mu^{-1}\|\btau\|_0 \quad \forall\btau\in [L^2(\Omega)]^{d\times d}.
\end{equation}

\begin{theorem}[reliability, compressible regime]\label{thm:reliability}
    Let $(\bu,\bsigma)$ solve the continuous problem \eqref{eq:mixed_elasticity}; let $(\bu_h,\bsigma_h)\in V_h\times S_h^g$ solve the discrete problem \eqref{eq:discrete_problem}; and let $(\br_h,\bw_h)\in V_h^{l+1}\times V_h^{l}$ solve \eqref{eq:local_postproc_mixed}. If Assumption~\ref{ass:saturation} is satisfied, then the following estimate holds:
    \begin{equation}\label{eq:reliability_standard}
        \mu^{-1/2}\|\bsigma-\bsigma_h\|_0+\mu^{1/2}\|\bu-\bw_h\|_h\lesssim\eta_h.
    \end{equation}
\end{theorem}
\begin{proof}
By Lemma~\ref{lemma:rel}, and after summing over all $K\in\CT_h$, we obtain
$$\|\varepsilon(\bz_h-\bw_h)\|_0=\|\varepsilon(\br_h)\|_0.$$
Thanks to this identity, the triangle inequality, and Assumption~\ref{ass:saturation},
\begin{equation}\label{eq:rel_disp}
    \|\varepsilon(\bu-\bw_h)\|_0\leq\frac{1}{1-\delta}\|\varepsilon(\br_h)\|_0.
\end{equation}

For the stress, using $\CC\bsigma-\varepsilon(\bu)=\bzero$,
$$\CC(\bsigma-\bsigma_h)=\varepsilon(\bu-\bw_h)-(\CC\bsigma_h-\varepsilon(\bw_h)),$$
so that, by the triangle inequality,
$$\|\CC(\bsigma-\bsigma_h)\|_0\leq\|\varepsilon(\bu-\bw_h)\|_0+\|\CC\bsigma_h-\varepsilon(\bw_h)\|_0.$$
Since $\bsigma-\bsigma_h=\mathcal{A}\big(\CC(\bsigma-\bsigma_h)\big)$ and $\|\mathcal{A}\btau\|_0\leq(2\mu+d\lambda)\|\btau\|_0$ (see \eqref{eq:A_bound}),
$$\|\bsigma-\bsigma_h\|_0\leq(2\mu+d\lambda)\big(\|\varepsilon(\bu-\bw_h)\|_0+\|\CC\bsigma_h-\varepsilon(\bw_h)\|_0\big).$$
Combined with \eqref{eq:rel_disp}, this yields
\begin{equation}\label{eq:rel_stress}
    \mu^{-1/2}\|\bsigma-\bsigma_h\|_0\lesssim\big(\mu^{1/2}+\lambda\mu^{-1/2}\big)\big(\|\varepsilon(\br_h)\|_0+\|\CC\bsigma_h-\varepsilon(\bw_h)\|_0\big).
\end{equation}

Finally, since $\bu\in [H^1_D(\Omega)]^d$ vanishes on $\Gamma_D$ and is single-valued across interior edges/faces, $\jmp{\bu-\bw_h}=\jmp{\bw_h}$ on $\CE_h^0$ and $\bu-\bw_h=-\bw_h$ on $\CE_h^D$; hence the jump and Dirichlet contributions to $\|\bu-\bw_h\|_h$ coincide exactly with those of $\bw_h$. Combining this with \eqref{eq:rel_disp} and \eqref{eq:rel_stress}, and recalling the definition of $\eta_K$, we obtain \eqref{eq:reliability_standard}.
\end{proof}

The reliability estimate of Theorem~\ref{thm:reliability} is not uniform with respect to $\lambda$: as discussed in Section~\ref{sec:elasticity_eqns}, the compliance tensor $\CC$ no longer controls the trace of the stress as $\lambda\to\infty$, and the constant in~\eqref{eq:reliability_standard} degenerates in this limit. This poses no difficulty away from the incompressible limit, where the constant in~\eqref{eq:reliability_standard} remains moderate and Theorem~\ref{thm:reliability} applies without restriction; however, as the material approaches incompressibility, an estimate valid uniformly in the Lam\'e parameters is required. Such an estimate relies on the following stability property of the continuous problem, with a constant that remains bounded as $\lambda\to\infty$.

\begin{theorem}[Robust stability]
\label{thm:robust-stability}
For any $(\bvarphi,\bpsi)\in[L^2(\Omega)]^{d\times d}_{\mathrm{sym}}\times[H^1_D(\Omega)]^d$, it holds that
\begin{equation}
\sup_{(\btau,\bv)\in[L^2(\Omega)]^{d\times d}_{\mathrm{sym}}\times[H^1_D(\Omega)]^d} \frac{\CB(\bvarphi,\bpsi;\btau,\bv)}{\mu^{-1/2}\|\btau\|_0+\mu^{1/2}\|\varepsilon(\bv)\|_0} \gtrsim \mu^{-1/2}\|\bvarphi\|_0+\mu^{1/2}\|\varepsilon(\bpsi)\|_0.
\end{equation}
\end{theorem}
\begin{proof}
See~\cite[Theorem~1]{LS2023}.
\end{proof}

To apply Theorem~\ref{thm:robust-stability}, we require a continuous, $H^1_D(\Omega)$-conforming counterpart of the postprocessed displacement $\bw_h$. To this end, let $\osw:V_h^{l}\to V_h^{l}\cap [H^1_D(\Omega)]^d$ denote the Oswald averaging operator, which satisfies the approximation property
\begin{equation}\label{eq:oswald_approx}
    \Bigg(\sum_{K\in\CT_h}h_K^{-2}\|\bv-\osw(\bv)\|_K^2\Bigg)^{1/2} \lesssim |\bv|_J \eq \Bigg(\sum_{E\in\CE_h^0} h_E^{-1}\|\jmp{\bv}\|_E^2\Bigg)^{1/2} \quad \forall\bv\in V_h^{l};
\end{equation}
see \cite[Theorem 2.1]{KarakashianPascal2007}. By the identity
\begin{equation*}
    \|\btau\|_0^2=\left\|\frac{1}{2}(\btau+\btau^\T)\right\|_0^2+\left\|\frac{1}{2}(\btau-\btau^\T)\right\|_0^2\quad\forall\btau\in[L^2(\Omega)]^{d\times d}
\end{equation*}
and an inverse inequality \cite[Lemma 1.44]{DiPietroErn2012}, \eqref{eq:oswald_approx} also yields the stability bound
\begin{equation}\label{eq:oswald_approx_eps}
    \|\varepsilon(\bv-\osw(\bv))\|_0 \leq \|\nabla(\bv-\osw(\bv))\|_0 \lesssim |\bv|_J \quad \forall\bv\in V_h^{l}.
\end{equation}
For the residual terms appearing in the reliability estimate, we introduce the oscillation quantities
\begin{equation}\label{eq:osc_f}
    \text{osc}(\bff) \eq \Bigg(\sum_{K\in\CT_h}h_K^2\|\bff-P_h\bff\|_K^2\Bigg)^{1/2},
\end{equation}
and
\begin{equation}\label{eq:osc_g}
    \text{osc}(\bg) \eq \Bigg(\sum_{E\in\CE_h^N}h_E\|\bg-Q_E\bg\|_E^2\Bigg)^{1/2},
\end{equation}
where $\CE_h^N$ denotes the set of edges (or faces) of $\CT_h$ lying on $\Gamma_N$. With these tools in hand, we now state the reliability estimate valid uniformly in the incompressible limit.

\begin{theorem}[Reliability, incompressible limit]
\label{thm:reliability-incompressible}
Let $(\bu,\bsigma)$ solve~\eqref{eq:mixed_elasticity}; let $(\bsigma_h,\bu_h)$ solve~\eqref{eq:discrete_problem}; and let $(\br_h,\bw_h)$ solve~\eqref{eq:local_postproc_mixed}. Then, it holds that
\begin{equation}
\mu^{-1/2}\|\bsigma-\bsigma_h\|_0 + \mu^{1/2}\|\bu-\bw_h\|_h \lesssim \eta_h + \mu^{-1/2}\big(\mathrm{osc}(\bff)+\mathrm{osc}(\bg)\big),
\label{eq:reliability-incompressible-cv}
\end{equation}
with a constant that remains bounded as $\lambda\to\infty$.
\end{theorem}
\begin{proof} 
By Theorem~\ref{thm:robust-stability}, applied with $(\bvarphi,\bpsi)=(\bsigma-\bsigma_h,\bu-\osw(\bw_h))$,
\begin{align*}
\mu^{-1/2}\|\bsigma-\bsigma_h\|_0\ +\ & \mu^{1/2}\|\varepsilon(\bu-\osw(\bw_h))\|_0\\
&\lesssim \mu^{1/2}\|\CC\bsigma_h-\varepsilon(\osw(\bw_h))\|_0 + \mu^{-1/2}\big(\mathrm{osc}(\bff)+\mathrm{osc}(\bg)\big).
\end{align*}
Indeed, using the definition of $\CB$, the equations satisfied by $(\bsigma,\bu)$, and the discrete problem, we have
\begin{align*}
\CB(\bsigma-\bsigma_h,\bu-\osw(\bw_h);\btau,\bv)
={}&(\CC\bsigma_h-\varepsilon(\osw(\bw_h)),\btau)\\
&+(\bff-P_h\bff,\bv)
+\langle\bg-Q_h\bg,\bv\rangle_{\Gamma_N}.
\end{align*}
Since $P_h$ is an $L^2$-projection,
\begin{align*}
|(\bff-P_h\bff,\bv)|
&=|(\bff-P_h\bff,\bv-P_h\bv)|\\
&\lesssim
\left(\sum_{K\in\mathcal T_h}
h_K^2\|\bff-P_h\bff\|_K^2\right)^{1/2}
\|\nabla\bv\|_0\lesssim
\mathrm{osc}(\bff)\|\varepsilon(\bv)\|_0,
\end{align*}
where the second inequality uses the approximation property of $P_h$, and the last follows from the second Korn inequality~\cite[Theorem~42.10]{ErnGuermondII}. Similarly,
\[
|\langle\bg-Q_h\bg,\bv\rangle_{\Gamma_N}|
\lesssim
\mathrm{osc}(\bg)\|\varepsilon(\bv)\|_0.
\]
By the triangle inequality and~\eqref{eq:oswald_approx_eps},
\begin{equation*}
\|\CC\bsigma_h-\varepsilon(\osw(\bw_h))\|_0 \leq \|\CC\bsigma_h-\varepsilon(\bw_h)\|_0 + |\bw_h|_J,
\end{equation*}
and, similarly, $\|\varepsilon(\bu-\bw_h)\|_0 \leq \|\varepsilon(\bu-\osw(\bw_h))\|_0 + |\bw_h|_J$. As in the proof of Theorem~\ref{thm:reliability}, the jump and Dirichlet contributions to $\|\bu-\bw_h\|_h$ coincide exactly with those of $\bw_h$. Combining the above,
\begin{align*}
\mu^{-1/2}\|\bsigma-\bsigma_h\|_0 + \mu^{1/2}&\|\bu-\bw_h\|_h\\
&\lesssim \mu^{1/2}\Big(\|\CC\bsigma_h-\varepsilon(\bw_h)\|_0 + |\bw_h|_J\Big) + \mu^{-1/2}\big(\mathrm{osc}(\bff)+\mathrm{osc}(\bg)\big).
\end{align*}
Since $\|\CC\bsigma_h-\varepsilon(\bw_h)\|_K^2$ and $|\bw_h|_J^2$ are each terms of the sum defining $\eta_K^2$, the result follows.
\end{proof}

\begin{remark}
Unlike the proof of Theorem~\ref{thm:reliability}, the residual $\br_h$ plays no essential role above. We nonetheless retain the same indicator $\eta_h$ in both regimes, so that a single, unified estimator drives mesh adaptivity regardless of the value of $\lambda$; the residual $\br_h$ will again be essential for the efficiency estimate of Theorem~\ref{thm:local_efficiency}.
\end{remark}

We conclude this section with a local efficiency estimate, showing that the indicator $\eta_K$ is bounded above by the local error, up to a constant that remains bounded as $\lambda\to\infty$.

\begin{theorem}[local efficiency]\label{thm:local_efficiency}
    Let $(\bu,\bsigma)$ solve the continuous problem \eqref{eq:mixed_elasticity}; let $(\bu_h,\bsigma_h)\in V_h\times S_h^g$ solve the discrete problem \eqref{eq:discrete_problem}; and let $(\br_h,\bw_h)\in V_h^{l+1}\times V_h^{l}$ solve \eqref{eq:local_postproc_mixed}. Then, the following estimate holds true for all $K\in\CT_h$, with a constant that remains bounded as $\lambda\to\infty$:
    \begin{equation}\label{eq:local_efficiency}
        \eta_K\lesssim\mu^{-1/2}\|\bsigma-\bsigma_h\|_K+\mu^{1/2}\|\bu-\bw_h\|_{h,K}.
    \end{equation}
\end{theorem}
\begin{proof}
Since $\CC\bsigma-\varepsilon(\bu)=\bzero$, the definitional equation \eqref{eq:local_postproc_mixed_a} of $\br_h$ gives, for all $\bv_K\in(I-P_h)V^{l+1}_h\big|_K$,
\begin{align*}
    (\varepsilon(\br_h),\varepsilon(\bv_K))_K &= (\CC\bsigma_h-\CC\bsigma,\varepsilon(\bv_K))_K + (\varepsilon(\bu)-\varepsilon(\bw_h),\varepsilon(\bv_K))_K\\
    &= -(\CC(\bsigma-\bsigma_h),\varepsilon(\bv_K))_K + (\varepsilon(\bu-\bw_h),\varepsilon(\bv_K))_K.
\end{align*}
Testing with $\bv_K=\br_h\big|_K\in(I-P_h)V^{l+1}_h\big|_K$ and applying the Cauchy-Schwarz inequality, together with~\eqref{eq:C_bound},
\begin{equation*}
    \mu^{1/2}\|\varepsilon(\br_h)\|_K^2 \leq \big(\mu^{1/2}\|\varepsilon(\bu-\bw_h)\|_K + \mu^{-1/2}\|\bsigma-\bsigma_h\|_K\big)\|\varepsilon(\br_h)\|_K,
\end{equation*}
whence
\begin{equation}\label{eq:eff_r_weighted}
    \mu^{1/2}\|\varepsilon(\br_h)\|_K \leq \mu^{1/2}\|\varepsilon(\bu-\bw_h)\|_K + \mu^{-1/2}\|\bsigma-\bsigma_h\|_K.
\end{equation}

Similarly, using $\CC\bsigma_h-\varepsilon(\bw_h) = -\CC(\bsigma-\bsigma_h)+\varepsilon(\bu-\bw_h)$, the triangle inequality and~\eqref{eq:C_bound} give
$$\|\CC\bsigma_h-\varepsilon(\bw_h)\|_K \leq \|\CC(\bsigma-\bsigma_h)\|_K + \|\varepsilon(\bu-\bw_h)\|_K \leq \mu^{-1}\|\bsigma-\bsigma_h\|_K + \|\varepsilon(\bu-\bw_h)\|_K,$$
whence
\begin{equation}\label{eq:eff_Csigma}
    \mu^{1/2}\|\CC\bsigma_h-\varepsilon(\bw_h)\|_K \leq \mu^{-1/2}\|\bsigma-\bsigma_h\|_K + \mu^{1/2}\|\varepsilon(\bu-\bw_h)\|_K.
\end{equation}

In addition, as in the proof of Theorem~\ref{thm:reliability}, the jump and Dirichlet contributions to $\eta_K$ coincide exactly with those of $\|\bu-\bw_h\|_{h,K}$.

Finally, adding \eqref{eq:eff_r_weighted} and \eqref{eq:eff_Csigma}, and using that the jump and Dirichlet terms in $\eta_K$ and $\|\bu-\bw_h\|_{h,K}$ coincide, we conclude that \eqref{eq:local_efficiency} holds.
\end{proof}

Hence, the estimate holds in both the compressible regime and the incompressible limit, in contrast to the reliability estimates of Theorems~\ref{thm:reliability} and~\ref{thm:reliability-incompressible}.

Notably, the proof above requires no auxiliary bubble functions: the test function $\bv_K=\br_h|_K$ used throughout is simply the residual itself, already available at no additional cost from the local residual minimization problem of Section~\ref{sec:postproc_minres}.

\begin{remark}[A single estimator for both regimes]
Both Theorem~\ref{thm:reliability} and Theorem~\ref{thm:reliability-incompressible} bound the same quantity, $\mu^{-1/2}\|\bsigma-\bsigma_h\|_0+\mu^{1/2}\|\bu-\bw_h\|_h$, by the same indicator $\eta_h$; only the dependence of the reliability constant on $\mu,\lambda$ differs between the two regimes. In~\cite{LS2023}, the compressible-regime estimate rests instead on the Prager-Synge hypercircle principle, an identity intrinsic to the norm $\|\cdot\|_\CC$ induced by the compliance tensor. Since $\|\cdot\|_\CC$ is no longer a norm as $\lambda\to\infty$, no such identity is available in the incompressible limit, and a second, distinct estimator, based on the continuous robust inf-sup estimate for the primal mixed formulation, is required for that regime. Here, by contrast, the transition between regimes is carried entirely by the elementary bound $\|\CA\btau\|_0\le(2\mu+d\lambda)\|\btau\|_0$, which remains valid, with a growing but finite constant, for every finite $\lambda$; a single indicator and comparison norm thus serve both regimes.
\end{remark}

\begin{remark}[Two mechanisms behind the same indicator]\label{rem:two-mechanisms}
The two theorems rely on different mechanisms: Theorem~\ref{thm:reliability} rests on Assumption~\ref{ass:saturation}, observed numerically in Section~\ref{sec:num-smooth} to hold for moderate $\lambda$ but to degrade as $\nu\to1/2$; Theorem~\ref{thm:reliability-incompressible}, in contrast, relies on the robust stability estimate of Theorem~\ref{thm:robust-stability} and requires no saturation assumption, making it the appropriate certificate of reliability precisely where Assumption~\ref{ass:saturation} is not guaranteed to hold.
\end{remark}

\begin{remark}[Trade-off with the hypercircle estimate]
This unification comes at a price: the hypercircle-based estimate of~\cite[Theorem~7]{LS2023} is a genuine identity, with effectivity index equal to one up to data oscillation, while the estimator introduced here, built from triangle and Cauchy-Schwarz inequalities, cannot attain this exact equality. The trade-off is one of unification and locality, a single, purely local construction valid across the whole range of compressibility, against the sharper, but regime-specific and globally coupled, hypercircle estimate.
\end{remark}

\section{Numerical examples}\label{sec:numerical_examples}

We validate the theoretical results of the previous sections with three numerical examples, implemented in Firedrake~\cite{FiredrakeUserManual}. Only one exactly symmetric mixed finite element family is considered: the Johnson-Mercier element (JM), which considers linear displacements and linear, $H({\rm div})$-conforming stresses on an Alfeld split. In order to implement the non-standard projection $P_h$, some observations in~\cite[Lemma 2]{JohnsonMercier1978}, along with the Slate package~\cite{slate} are used. In turn, postprocesses are obtained by enforcing a Lagrange multiplier and by using the projection $P_h$.

The displacement space is given by discontinuous linear elements, i.e.
\begin{equation}
V_h = \{\bv\in[L^2(\Omega)]^d \mid \bv|_K \in [P_1(K)]^d \ \ \forall K\in\CT_h\}.
\end{equation}
For each example, we report the relative errors
\begin{equation}
e_0^{\bsigma} = \frac{\|\bsigma-\bsigma_h\|_0}{\|\bsigma\|_0}, \qquad
e_h^{\bu} = \frac{\|\bu-\bw_h\|_h}{\|\bu\|_h},
\end{equation}
and the effectivity index
\begin{equation}
c_{\text{eff}} = \frac{\eta_h}{\mu^{-1/2}\|\bsigma-\bsigma_h\|_0+\mu^{1/2}\|\bu-\bw_h\|_h}.
\end{equation}
When the exact solution is unavailable, $\bsigma$ and $\bu$ are replaced in these definitions by a reference solution computed on a uniformly refined mesh. Adaptive computations follow the standard loop
\[
\text{SOLVE} \to \text{ESTIMATE} \to \text{MARK} \to \text{REFINE} \to \text{SOLVE} \to \cdots,
\]
with D\"orfler marking~\cite{Dorfler1996}: recalling the local indicator $\eta_K$ of~\eqref{eq:etaK}, an element $K\in\CT_h$ is refined whenever $\eta_K^2\ge\theta\max_{K'\in\CT_h}\eta_{K'}^2$, $\theta=1/4$. We denote by $N$ the number of elements of $\CT_h$. In what follows, we specify the material by its Young's modulus $E$ and Poisson's ratio $\nu\in(0,1/2)$, related to the Lam\'e parameters by
\begin{equation}
\lambda = \frac{E\nu}{(1+\nu)(1-2\nu)}, \qquad \mu = \frac{E}{2(1+\nu)}.
\end{equation}
The incompressible limit $\lambda\to\infty$ corresponds to $\nu\to1/2$.

\subsection{A smooth solution}
\label{sec:num-smooth}

To verify the predicted convergence rates, we let $\Omega=(0,1)^2$, $\Gamma_D=\partial\Omega$, $\Gamma_N=\emptyset$, and take the exact displacement
\begin{equation}
u_1(x,y) = \sin(\pi x)\sin(\pi y), \qquad u_2(x,y) = -\sin(\pi x)\sin(\pi y),
\end{equation}
with body force $\bff=-\ddiv\bsigma$ and stress $\bsigma=\CC^{-1}\varepsilon(\bu)$. Since $\ddiv\bu\not\equiv0$, the exact stress depends genuinely on the Lam\'e parameters, so this example already probes the incompressible limit despite the absence of singularities. We fix $E=1$.

A single indicator $\eta_h$, and a single comparison norm, are used throughout, both in the compressible regime and in the incompressible limit. 

As noted following the proof of Theorem~\ref{thm:a_priori}, for a sufficiently smooth exact solution the expected optimal rate for the stress in the $L^2$-norm is $\CO(h^{k+1})$ generically, and $\CO(h^2)$ for the JM method, whose stress and displacement spaces are both piecewise linear; by Corollary~\ref{cor:postproc_bound_minres}, the postprocessed displacement $\bw_h$ attains the same order in the broken energy norm $\|\cdot\|_h$, a superconvergent rate for this quantity. Table~\ref{tab:smooth} confirms that the expected rate is attained for both values of $\nu$.

In particular, the optimal rate is attained without degradation as $\nu\to1/2$, confirming that the method is free of the locking phenomenon that affects standard primal discretizations in the incompressible limit. The larger magnitude of $e_h^{\bu}$ observed for $\nu=0.4999$ is consistent with the theory: the quantity controlled uniformly in $\lambda$ by Theorem~\ref{thm:reliability-incompressible} is the weighted combination $\mu^{-1/2}\|\bsigma-\bsigma_h\|_0+\mu^{1/2}\|\bu-\bw_h\|_h$, rather than $\|\bu-\bw_h\|_h$ alone, and $\mu^{1/2}\to0$ as $\nu\to1/2$.

\begin{table}[h]
\centering
\small
\caption{Smooth example: relative errors and orders of convergence for JM elements, uniform refinement, $\nu=0.3$ and $\nu=0.4999$.}
\label{tab:smooth}
\begin{tabular}{lrccccccc}
\hline
& $N$ & $e_0^{\bsigma}$ & (oc) & $e_h^{\bu}$ & (oc) & $\eta_h$ & (oc) & $c_{\rm eff}$\\
\hline
\multirow[t]{6}{*}{$\nu=0.3$}
& 2134 & 1.54e-02 & - & 2.49e-02 & - & 5.15e-02 & - & 0.381 \\
 & 9248 & 3.67e-03 & 1.85 & 5.67e-03 & 1.90 & 1.20e-02 & 1.87 & 0.379 \\
 & 35986 & 9.42e-04 & 1.93& 1.43e-03 & 1.95 & 3.08e-03 & 1.93 & 0.380 \\
  & 143192 & 2.38e-04 & 1.94 & 3.60e-04 & 1.95 & 7.77e-04 & 1.94 & 0.380 \\
   & 570094 & 5.96e-05 & 1.97 & 9.00e-05 & 1.98 & 1.95e-04 & 1.97 & 0.381 \\
   & 2274938 & 1.49e-05 & 1.98 & 2.25e-05 & 1.98 & 4.88e-05 & 1.98 & 0.381 \\
\hline

\multirow[t]{5}{*}{$\nu=0.4999$}
& 2134 & 1.49e-02 & - & 3.32e+01 & - & 6.28e+01 & - & 0.360 \\
& 9248 & 3.76e-03 & 1.77 & 8.11e+00 & 1.82 & 1.58e+01 & 1.78 & 0.364 \\
& 35986 & 9.32e-04 & 1.97 & 1.98e+00 & 2.00 & 3.91e+00 & 1.98 & 0.365 \\
& 143192 & 2.33e-04 & 1.95 & 4.92e-01 & 1.96 & 9.78e-01 & 1.95 & 0.366 \\
& 570094 & 5.80e-05 & 1.98 & 1.22e-01 & 1.99 & 2.43e-01 & 1.98 & 0.366 \\
& 2274938 & 1.45e-05 & 1.99 & 3.04e-02 & 1.99 & 6.07e-02 & 1.99 & 0.366 \\
\hline
\end{tabular}
\end{table}

Since the exact solution is known, this example also allows us to verify the behaviour of the saturation constant $\delta$ of Assumption~\ref{ass:saturation}, used in the proof of Theorem~\ref{thm:reliability}. Table~\ref{tab:saturation} reports the computed ratio
\begin{equation}
\delta_h(\nu) \eq \frac{\|\varepsilon(\bu-\bz_h)\|_0}{\|\varepsilon(\bu-\bw_h)\|_0}
\end{equation}
for JM, across a range of mesh sizes, characterized equivalently by the number of elements $N$ or the mesh parameter $h\sim N^{-1/2}$, and Poisson ratios $\nu$: the ratio remains bounded away from one for $\nu\le0.3$, but approaches, and eventually exceeds, one as $\nu$ increases further, reaching or surpassing this threshold already at $\nu=0.35$ (equivalently, $h$). This is consistent with the two-mechanism picture of Remark~\ref{rem:two-mechanisms}: reliability in this more compressible regime is certified by Theorem~\ref{thm:reliability}, under Assumption~\ref{ass:saturation}; beyond the point where the saturation assumption ceases to hold, reliability of the same indicator $\eta_h$ is instead certified by Theorem~\ref{thm:reliability-incompressible}, which does not rely on this assumption.

\begin{table}[h]
\centering
\small
\caption{Smooth example: computed saturation ratio $\delta_h(\nu)$, JM method, for varying $N$ and $\nu$.}
\label{tab:saturation}
\begin{tabular}{rcccc}
\hline
$N$ & $\nu=0.2$ & $\nu=0.25$ & $\nu=0.3$ & $\nu=0.35$ \\
\hline
$2134$ & $0.90$ & $0.92$ & $0.95$ & $0.98$ \\
$9248$ & $0.91$ & $0.94$ & $0.97$ & $\geq 1$ \\
$35986$ & $0.93$ & $0.95$ & $0.98$ & $\geq 1$ \\
$143192$ & $0.94$ & $0.96$ & $0.99$ & $\geq 1$ \\
$570094$ & $0.95$ & $0.96$ & $0.99$ & $\geq 1$ \\
\hline
\end{tabular}
\end{table}

Figure~\ref{fig:smooth-ceff} shows $c_{\text{eff}}=\eta_h/\big(\mu^{-1/2}\|\bsigma-\bsigma_h\|_0+\mu^{1/2}\|\bu-\bw_h\|_h\big)$ as a function of $\nu$, for $\nu\in\{0.3,0.4,0.45,0.49,0.499,0.4999\}$: the index remains uniformly bounded in $\nu$, empirically confirming the robustness guaranteed by Theorem~\ref{thm:reliability-incompressible} in the entire range of compressibility.

\begin{figure}[h]
\centering
\includegraphics[width=0.6\textwidth]{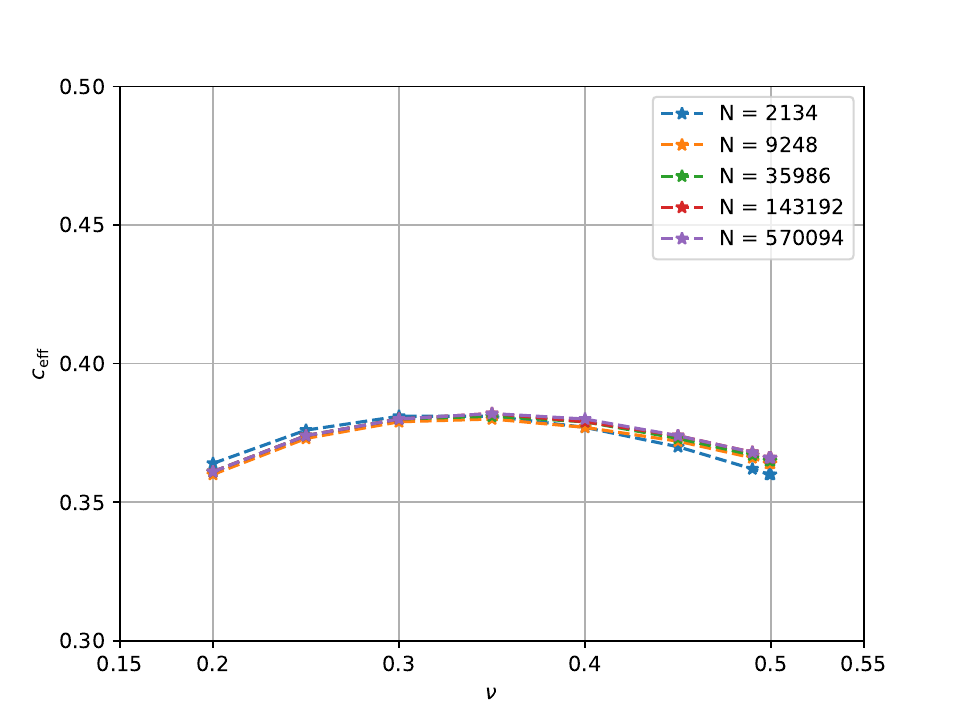}
\caption{Smooth example: effectivity index $c_{\text{eff}}$ as a function of $\nu$.}
\label{fig:smooth-ceff}
\end{figure}

\subsection{An L-shaped domain}
\label{sec:num-Lshape}

Next, we consider the benchmark of~\cite[Section 8.2]{LS2023}, with a known singular solution. The domain
\[
\Omega = \{(x,y): |x|+|y|\le \sqrt2\,a\} \setminus \{(x,y): |x+a/\sqrt2|+|y|\le a/\sqrt2\},
\]
with $a=1$, has a reentrant corner of interior angle $3\pi/2$, and traction boundary conditions are imposed on the whole boundary, $\Gamma_N=\Gamma$, $\Gamma_D=\emptyset$. Since $\Gamma=\Gamma_N$, the displacement $\bu$ is determined only up to a rigid body motion. Consequently, a Lagrange multiplier is introduced in the implementation to enforce uniqueness. In polar coordinates centered at the reentrant corner, the exact displacement is
\begin{align*}
u_x &= \frac{1}{2\mu}r^{\alpha}\big((\kappa-Q(\alpha+1))\cos(\alpha\theta)-\alpha\cos((\alpha-2)\theta)\big),\\
u_y &= \frac{1}{2\mu}r^{\alpha}\big((\kappa+Q(\alpha+1))\sin(\alpha\theta)+\alpha\sin((\alpha-2)\theta)\big),
\end{align*}
with $\alpha=0.544483737$, $Q=0.543075579$, $\kappa=3-4\nu$. We fix $E=1$ and consider $\nu=0.3$ and $\nu=0.4999$.

Since $\bu\in H^{1+\alpha}(\Omega)$ only, with $\alpha<0.55$~\cite{SzaboBabuska1991}, uniform refinement yields the suboptimal rate $\CO(N^{-\alpha/2})$, independently of the polynomial degree. Figure~\ref{fig:Lshape-conv} compares the convergence history of $\eta_h$ under adaptive refinement, driven by the indicator of Section~\ref{sec:aposteriori}, against uniform refinement, for the JM method and both values of $\nu$: adaptivity recovers the optimal rate $\CO(N^{-k/2})$ in each case. Figure~\ref{fig:Lshape-mesh} displays a representative adaptively refined mesh, confirming that refinement concentrates at the reentrant corner.

\begin{figure}[h]
\centering
 \includegraphics[width=0.49\textwidth]{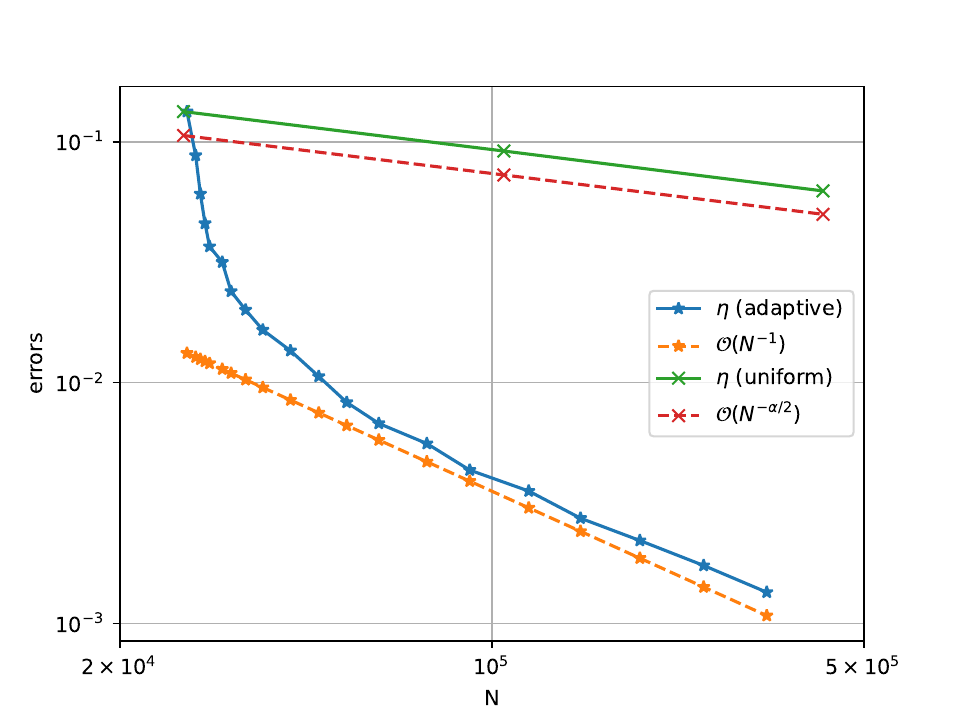}
 \includegraphics[width=0.49\textwidth]{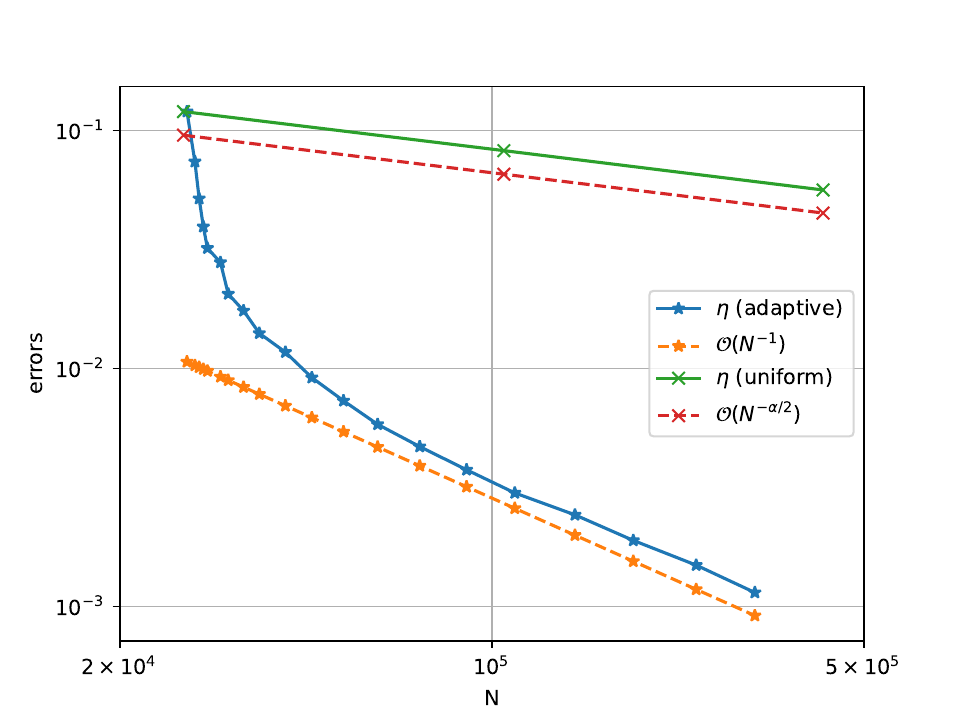}
\caption{L-shaped domain: convergence history of $\eta_h$ under adaptive and uniform refinement, $\nu=0.3$ (left) and $\nu=0.4999$ (right).}
\label{fig:Lshape-conv}
\end{figure}

\begin{figure}[h]
\centering
 \includegraphics[width=0.8\textwidth]{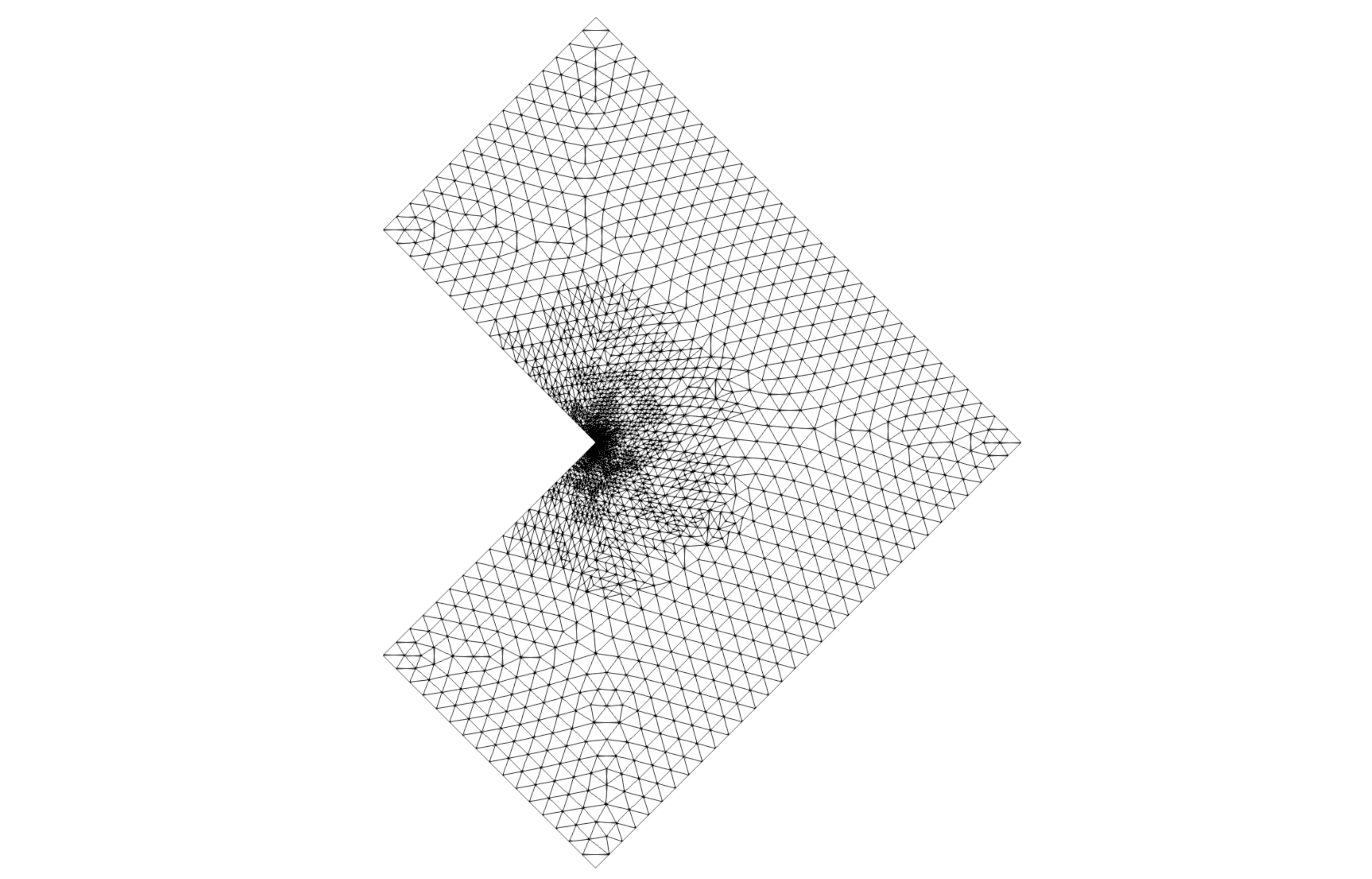}
\caption{L-shaped domain: adaptively refined mesh for $\nu=0.3$ and $N=90791$ after 15 refinement steps.}
\label{fig:Lshape-mesh}
\end{figure}

\subsection{Cook's membrane}
\label{sec:num-cook}

Our last example is Cook's membrane, a standard benchmark in computational mechanics for which no analytical solution is available. The domain is the trapezoid with vertices $(0,0)$, $(48,44)$, $(48,60)$, $(0,44)$, clamped on $\Gamma_D=\text{conv}\{(0,0),(0,44)\}$, subject to a shear traction $\bg=(0,1)$ on $\text{conv}\{(48,44),(48,60)\}$ and traction-free elsewhere, and $\bff=(0,0)$. We take $E=1$ and $\nu=0.4999$, so that the material is close to the incompressible limit. This is the most demanding regime for this problem, owing to the incompatible traction data at the corners $(48,44)$ and $(48,60)$: $\bg$ does not belong to the trace space of $S_h$ on any triangulation, so that the oscillation term $\text{osc}(\bg)$ does not vanish.

Figure~\ref{fig:cook-conv} shows the convergence history of $\eta_h$ under adaptive versus uniform refinement, with the reference solution taken on a uniformly refined mesh. Figure~\ref{fig:cook-mesh} shows the resulting adaptively refined mesh, with refinement concentrated near the corners where the traction data is incompatible.

\begin{figure}[h!]
\centering
\includegraphics[width=0.6\textwidth]{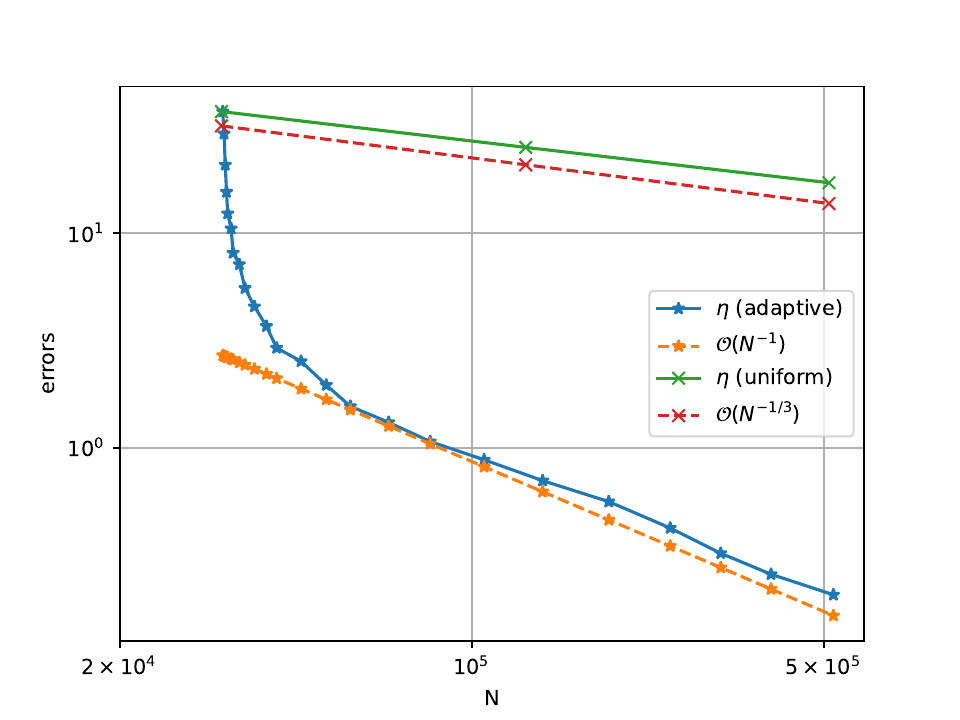}
\caption{Cook's membrane: convergence history of $\eta_h$ under adaptive and uniform refinement.}
\label{fig:cook-conv}
\end{figure}

\begin{figure}[h!]
\centering
\includegraphics[width=0.8\textwidth]{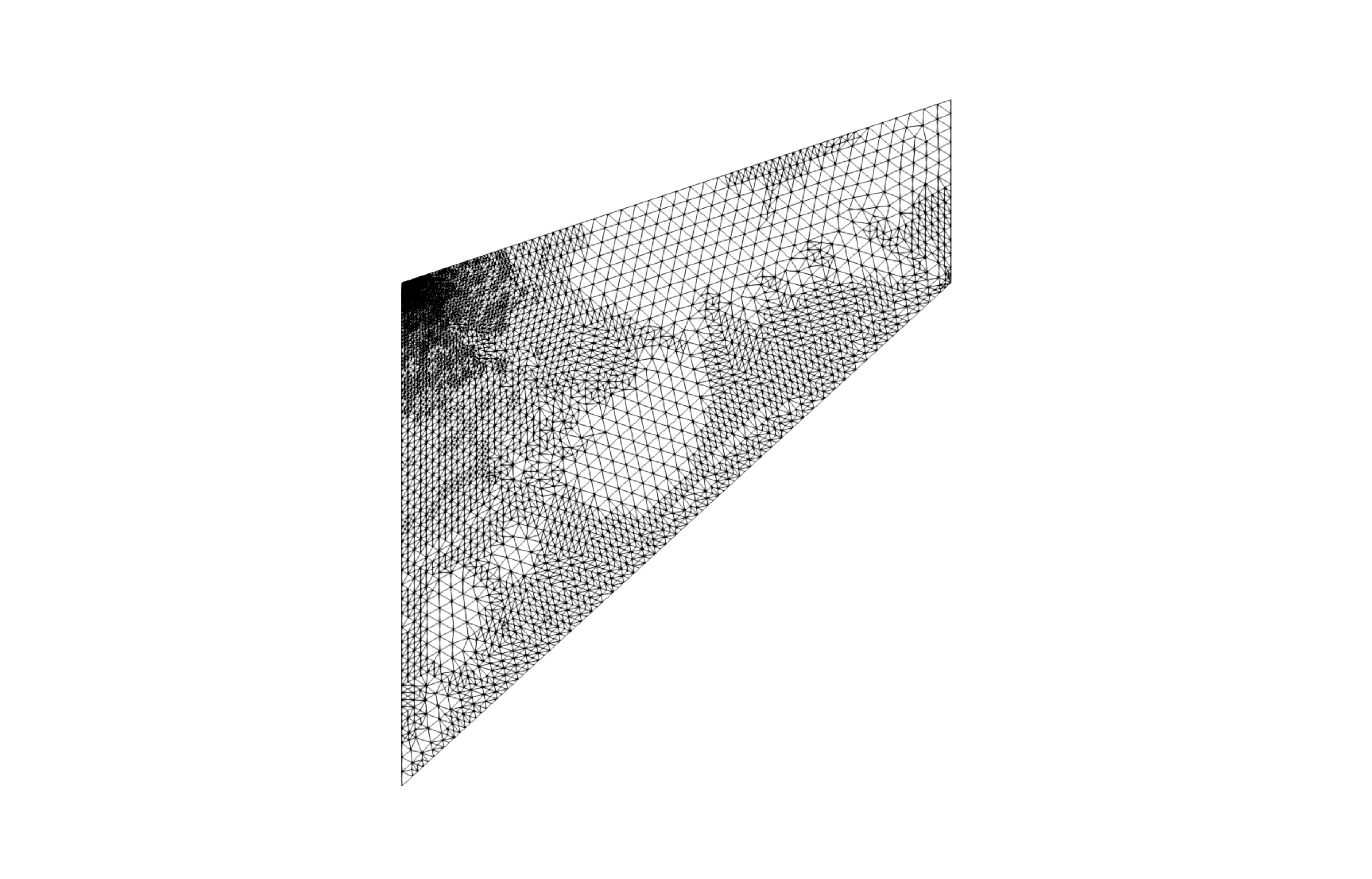}
\caption{Cook's membrane: adaptively refined mesh after 15 refinements.}
\label{fig:cook-mesh}
\end{figure}

\section{Conclusions}
\label{sec:conclusions}

We have developed an a posteriori error estimator for exactly symmetric mixed finite element discretizations of linear elasticity, based on recasting a Stenberg-type postprocessing scheme as a local residual minimization problem. The resulting estimator is reliable and locally efficient in the compressible regime, and admits an alternative reliability estimate, with a constant that remains bounded as $\lambda\to\infty$.

A central feature of our approach, illustrated numerically in Section~\ref{sec:numerical_examples}, is that a single indicator and a single comparison norm serve both regimes, in contrast with the hypercircle-based estimator of~\cite{LS2023}, which requires two distinct constructions: the Prager-Synge hypercircle principle for the compressible regime, and the continuous robust stability estimate of the primal mixed formulation for the incompressible limit, owing to the fact that $\|\cdot\|_\CC$ is no longer a norm as the material approaches incompressibility. This unification comes at the cost of an effectivity index that, unlike the exact hypercircle identity, does not attain the value one: our estimator is built from triangle and Cauchy-Schwarz inequalities rather than an exact energy identity. In exchange, it is entirely local, requiring no additional global solve, and is embedded within the same residual-minimization framework already applied to scalar diffusion problems~\cite{MugaRojasVega2023} and to the Helmholtz equation~\cite{CamargoRojasVega2025}, suggesting a natural path toward further extensions.

In a forthcoming companion paper, we address, following the same methodology, the case of weakly symmetric mixed finite elements. The present framework may also extend to other settings admitting a Stenberg-type postprocessing scheme, such as fourth-order elliptic problems, or the obstacle problem, where the underlying variational inequality is resolved at the discretization stage; adapting the residual minimization framework to the resulting discrete approximations, however, would require its own analysis, in particular regarding the postprocessing scheme and the saturation assumption underlying the reliability estimates.

%\section*{Acknowledgments}
\newpage
\bibliographystyle{siamplain}
\bibliography{references}
\end{document}

%% file: ex_shared.tex
\usepackage{lipsum}
\usepackage{amsfonts}
\usepackage{graphicx}
\usepackage{epstopdf}
\usepackage{algorithmic}
\ifpdf
  \DeclareGraphicsExtensions{.eps,.pdf,.png,.jpg}
\else
  \DeclareGraphicsExtensions{.eps}
\fi

\newsiamremark{remark}{Remark}
\newsiamremark{hypothesis}{Hypothesis}
\crefname{hypothesis}{Hypothesis}{Hypotheses}
\newsiamthm{claim}{Claim}
\newsiamremark{fact}{Fact}
\crefname{fact}{Fact}{Facts}

\headers{Minimum-residual error indicator for elasticity}{E. C\'aceres and P. Vega}

\title{An adaptive superconvergent mixed finite element method for exactly symmetric linear elasticity based on local residual minimization\thanks{Submitted to the editors DATE.
\funding{The first author acknowledges the support provided by the Dirección de Investigación y Desarrollo (DICYT) of Universidad de Santiago de Chile under project DICYT Postdoctorado 042632VR\_Postdoc, and previously by the Agencia Nacional de Investigación y Desarrollo (ANID) through project FONDECYT 11251691. The second author gratefully acknowledges the support from ANID through project FONDECYT 11251691.}}}

\author{Ernesto Cáceres\thanks{Department of Mathematics and Computer Science; Computational Heat and Fluid Flow\\ Laboratory, Universidad de Santiago de Chile, Santiago, Chile 
  (\email{ernesto.caceres@usach.cl}).}
\and Patrick Vega\thanks{Department of Mathematics and Computer Science; Computational Heat and Fluid Flow\\ Laboratory, Universidad de Santiago de Chile, Santiago, Chile 
  (\email{patrick.vega@usach.cl}).}}

\usepackage{amsopn}

%% file: general_commands.tex
\usepackage{amssymb}
\usepackage{mathrsfs}

\newcommand{\eq}{:=}
\newcommand{\osw}{I_{\mathrm{av}}^h}

\newcommand{\bsigma}{\boldsymbol{\sigma}}
\newcommand{\btau}{\boldsymbol{\tau}}
\newcommand{\bff}{\boldsymbol{f}}
\newcommand{\bzero}{\boldsymbol{0}}
\newcommand{\bpsi}{\boldsymbol{\psi}}
\newcommand{\bvarphi}{\boldsymbol{\varphi}}

\DeclareMathOperator*{\argmin}{arg\,min}

\newcommand{\ddiv}{\operatorname{div}}

\newcommand{\jmp}[1]{[\![#1]\!]}

\newcommand{\bg}{\boldsymbol g}

\newcommand{\bn}{\boldsymbol n}

\newcommand{\br}{\boldsymbol r}

\newcommand{\bu}{\boldsymbol u}
\newcommand{\bv}{\boldsymbol v}
\newcommand{\bw}{\boldsymbol w}

\newcommand{\bz}{\boldsymbol z}

\newcommand{\CA}{\mathcal A}
\newcommand{\CB}{\mathcal B}
\newcommand{\CC}{\mathcal C}

\newcommand{\CE}{\mathcal E}

\newcommand{\CM}{\mathcal M}

\newcommand{\CO}{\mathcal O}

\newcommand{\CT}{\mathcal T}

\newcommand{\T}{\mathsf{T}}